\documentclass{amsart}
\usepackage[margin=1in]{geometry}
\usepackage[english]{babel}
\usepackage{amsmath, amsthm, amssymb, amsfonts,tikz,mathtools,comment}
\usepackage{multicol}
\usepackage{float}
\newtheorem{theorem}{Theorem}[section]

\newtheorem{lemma}[theorem]{Lemma}
\newtheorem{conjecture}[theorem]{Conjecture}
\newtheorem{prop}[theorem]{Proposition}

\theoremstyle{definition}

\DeclarePairedDelimiter{\ceil}{\lceil}{\rceil}
\DeclarePairedDelimiter{\floor}{\lfloor}{\rfloor}

\newcommand{\dom}{\partial}

\title{Toward a Nordhaus-Gaddum Inequality for the Number of Dominating Sets}
\author{Lauren Keough and David Shane}
\date{}

\begin{document}

\maketitle

\begin{abstract}
A dominating set in a graph $G$ is a set $S$ of vertices such that every vertex of $G$ is either in $S$ or is adjacent to a vertex in $S$. Nordhaus-Gaddum inequailties relate a graph $G$ to its complement $\bar{G}$.  In this spirit Wagner proved that any graph $G$ on $n$ vertices satisfies $\dom(G)+\dom(\bar{G})\geq 2^n$ where $\dom(G)$ is the number of dominating sets in a graph $G$. In the same paper he comments that  an upper bound for $\dom(G)+\dom(\bar{G})$ among all graphs on $n$ vertices seems to be much more difficult.  Here we prove an upper bound on $\dom(G)+\dom(\bar{G})$ and prove that any graph maximizing this sum has minimum degree at least $\floor{n/2}-2$ and maximum degree at most $\ceil{n/2}+1$.  We conjecture that the complete balanced bipartite graph maximizes $\dom(G)+\dom(\bar{G})$ and have verified this computationally for all graphs on at most $10$ vertices.
\end{abstract}

\section{Introduction}

A \emph{dominating set} in a graph $G$ is a set of vertices $S$ such that every vertex of $G$ is either in $S$ or adjacent to a vertex in $S$.  Dominating sets, and their many variations, have long been studied \cite{DS-book}.  Also long studied are Nordhaus-Gaddum inequalities which describe the relationship between a graph parameter on $G$ and the same graph parameter on $\bar{G}$, the complement of $G$, in terms of the order of the graph.  The original Nordhaus-Gaddum inequalities concern the chromatic number of a graph $G$, denoted $\chi(G)$.  In  \cite{NG}, Nordhaus and Gaddum prove that, if $G$ has $n$ vertices, 
\[2\sqrt{n}\leq\chi(G) + \chi(\bar{G})\leq n+1\] 
and
\[n\leq \chi(G)\cdot\chi(\bar{G})\leq \left( \frac{n+1}{2}\right)^2.\]
Since then there have been several hundred papers proving similar relations for many different graph parameters \cite{NG-General}. In particular, there are such inequalities for the domination number (the size of a smallest dominating set) \cite{NG-DominationNumber, NG-DominationNumber2}.  
  See \cite{NG-General} and \cite{NG-domination} for surveys of results concerning Nordhaus-Gaddum inequalities for at least 30 types of domination numbers.

Separately, there has been interest in results concerning maximizing or minimizing the \emph{number} of a given graph substructure, rather than their size, subject to certain conditions.  
For a survey on these types of problems for regular graphs see \cite{EGT-ZhaoSurvey}. 
Recently, there have been several papers that maximize or minimize the total number of dominating sets or total dominating sets for connected graphs of a given order \cite{EGT-domsets, Wagner, EGT-domsets2, EGT-totaldomsets}.

Let $\partial(G)$ be the number of dominating sets in a graph $G$.  Uniting the ideas of Nordhaus-Gaddum inequalities and counting the number of graph substructures, Wagner \cite{Wagner} proves that
	\[ \partial(G)+\partial(\bar{G}) \geq 2^n .\]
In the same paper, he proposes that determining the maximum of $\partial(G) + \partial(\bar{G})$ as $G$ ranges over all possible graphs on $n$ vertices seems to be much more difficult.  We are able to prove the following theorem.

\begin{theorem}\label{thm:main}
If $G$ is a graph on $n$ vertices, then
\begin{equation*}
    \partial (G) + \partial (\bar{G})\leq 2^{n+1}-2^{\floor{\frac{n}{2}}}-2^{\ceil{\frac{n}{2}}-1}.
\end{equation*}
\end{theorem}

However, this is not the least upper bound. The  authors and Wagner conjecture that the extremal graph is the complete balanced bipartite graph, leading to the following conjecture.

\begin{conjecture}\label{conj} For a graph $G$ on $n$ vertices,
	\[\dom(G) + \dom(\bar{G}) \leq 2(2^{\floor{\frac{n}{2}}}-1) (2^{\ceil{\frac{n}{2}}} -1) + 2=\partial\left(K_{\floor{\frac{n}{2}},{\ceil{\frac{n}{2}}}}\right)+\partial\left(\overline{K_{\floor{\frac{n}{2}},{\ceil{\frac{n}{2}}}}}\right).\]
\end{conjecture}
This conjecture has been verified up to $n=10$ vertices. Wagner points out that this conjecture makes heuristic sense as both the complete balanced bipartite graph and its complement can be dominated by only two vertices (personal communication, October 3, 2017). 

Throughout the paper we use $N_G(v)$ to mean the open neighborhood of the vertex $v$ in the graph $G$ and $N_G[v]$ for the closed neighborhood of $v$ in $G$.  If $S$ is a set of vertices we define $N_G(S)$ and $N_G[S]$ similarly. In Section \ref{sec:MainResult} we prove Theorem \ref{thm:main}.  In Section \ref{sec:DegreeCondition} we provide a maximum and minimum degree condition for the extremal graph. Finally, in Section \ref{sec:Conclusion} we provide some asymptotics and describe some of the difficulties in finding the least upper bound for $\dom(G) + \dom(\bar{G})$.

\section{An Upper Bound for $\dom(G)+\dom(\bar{G})$}\label{sec:MainResult}
To prove that $\dom(G)+\dom(\bar{G})\ge 2^n$, Wagner uses the fact that if a set $S$ does not dominate $G$, then $\bar{S}$ dominates $\bar{G}$ \cite{Wagner}. We use this same fact to express the sum of the number of dominating sets in $G$ and $\bar{G}$ as
\begin{equation*}{\label{eq:upsilondef}}
    \partial (G)+\partial(\bar{G})=2^n+\Upsilon(G,\bar{G})
\end{equation*}
where
\begin{equation*}
    \Upsilon(G,\bar{G})=|\,\{\,A\subseteq V(G) : A \text{ dominates } G \text{ and } \bar{A} \text{ dominates } \bar{G}\, \}\,|.
\end{equation*}
We make use of $\Upsilon(G,\bar{G})$ to establish the following upper bound.
\begin{lemma}{\label{lem:k}}
If $G$ is a graph on $n$ vertices and a vertex $v\in V(G)$ has $\deg_G(v)=k$, then
\begin{equation*}
    \partial (G) + \partial (\bar{G})\leq 2^{n+1}-2^{k}-2^{n-k-1}.
\end{equation*}
\end{lemma}
\begin{proof}
We bound $\Upsilon(G,\bar{G})$ in terms of $n$ and $k$ and consequentially bound $\partial(G)+\partial(\bar{G})$ in terms of $n$ and $k$. 
It will be helpful to visualize $G$ and $\bar{G}$ as shown in Figure \ref{fig:arbunion}. Note that the graphs in Figure \ref{fig:arbunion} do not include any edges that are not incident with $v$, but every edge is in either $G$ or $\bar{G}$.

\begin{figure}[H]
    \centering
\begin{tikzpicture}
  [scale=.8,auto=left,every node/.style={circle,fill=blue!20}]
  \node[fill=black,scale=.5] (n1) [label=below:{$v$}] at (3,-.5) {};
  \node[fill=black,scale=.5] (n3)  [] at (0,1) {};
    \node[fill=black,scale=.2] (n7) at (.8,1){};
    \node[fill=black,scale=.2] (n7) [] at (1,1){};
    \node[fill=white,scale=.2] (n7) [label=above:{$N_G(v)$}] at (1,.8){};
    \node[fill=white,scale=.2] (n9) at (1,.9){};
    \node[fill=black,scale=.2] (n7) at (1.2,1){};
  \node[fill=black,scale=.5] (n4)  [] at (2,1) {};
  \node[fill=black,scale=.5] (n7)  [] at (4,1) {};
    \node[fill=black,scale=.2] (n7) at (4.8,1){};
    \node[fill=black,scale=.2] (n7) [] at (5,1){};
    \node[fill=white,scale=.2] (n7) [label=above:{$\overline{N_G[v]}$}] at (5,.8){};
    \node[fill=black,scale=.2] (n7) at (5.2,1){};
  \node[fill=black,scale=.5] (n7)  [] at (6,1) {};

  \node[fill=black,scale=.5] (n2) [label=below:{$v$}] at (13,-.5) {};
  \node[fill=black,scale=.5] (n7)  [] at (10,1) {};
    \node[fill=black,scale=.2] (n7) at (10.8,1){};
    \node[fill=black,scale=.2] (n7) [] at (11,1){};
    \node[fill=white,scale=.2] (n7) [label=above:{$\overline{N_{\bar{G}}[v]}$}] at (11,.8){};
    \node[fill=white,scale=.2] (n11) at (15,.9){};
    \node[fill=black,scale=.2] (n7) at (11.2,1){};
  \node[fill=black,scale=.5] (n7)  [] at (12,1) {};
  \node[fill=black,scale=.5] (n8)  [] at (14,1) {};
    \node[fill=black,scale=.2] (n7) at (14.8,1){};
    \node[fill=black,scale=.2] (n7) [] at (15,1){};
    \node[fill=white,scale=.2] (n7) [label=above:{$N_{\bar{G}}(v)$}] at (15,.8){};
    \node[fill=black,scale=.2] (n7) at (15.2,1){};
  \node[fill=black,scale=.5] (n10)  [] at (16,1) {};
  
      \node[fill=white,scale=.2] (n7) [label=above:{$G$}] at (3,2.5){};
      \node[fill=white,scale=.2] (n7) [label=above:{$\bar{G}$}] at (13,2.5){};
  
  \foreach \from/\to in {n1/n3,n1/n4,n2/n8,n2/n10}
    \draw (\from) edge  (\to);
    

 \draw (3,1.35) ellipse (4cm and 1cm);
 \draw (3,.35) -- (3,2.35);
 
 \draw (13,1.35) ellipse (4cm and 1cm);
 \draw (13,.35) -- (13,2.35);

\end{tikzpicture}
    \caption{A drawing of $G$ and $\bar{G}$ to aid in the proof of Lemma \ref{lem:k}.}
    \label{fig:arbunion}
\end{figure}
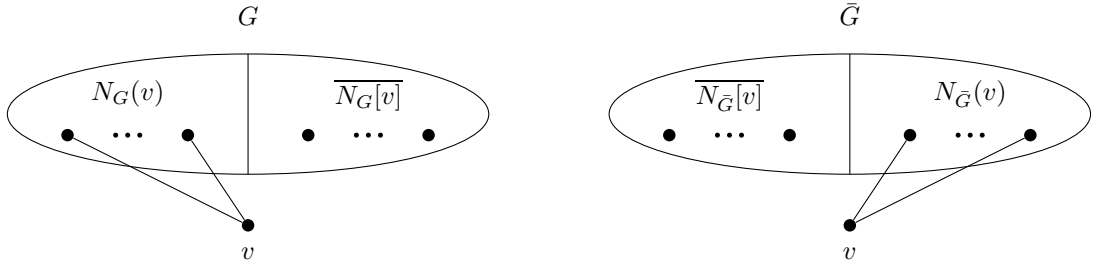

Let's consider a set $S\subseteq V(G)$ with the following properties:
\begin{itemize}
    \item $v\in S$.
    \item $N_{\bar{G}}(v)=\overline{N_G[v]}\subseteq S$.
\end{itemize}
We claim that $\bar{S}$ is not a dominating set of $\bar{G}$.  Since $\bar{S}\cap N_{{\bar G}}(v) = \emptyset$ and $v\notin \bar{S}$, $v\notin N_{\bar{G}}[S]$. Thus, $\bar{S}$ is not a dominating set of $\bar{G}$.
Therefore all sets satisfying the construction of $S$ are not counted in $\Upsilon(G,\bar{G})$. Since each element of $N_G(v)$ may or may not be included in $S$ and $|N_G(v)|=\deg_G(v)=k$, we have identified $2^k$ sets that are not  in $\Upsilon(G,\bar{G})$.

Let's now consider a set $T\subseteq V(G)$ with the following properties:
\begin{itemize}
    \item $v\notin T$  
    \item $T\cap N_G(v) = \emptyset$
\end{itemize}

Since $v\notin N_G[T]$, $T$ is not a dominating set of $G$ and all sets satisfying the construction of $T$ are not counted in $\Upsilon(G,\bar{G})$.  Since each element of $N_{\bar{G}}(v)$ may or may not be included in $T$ and $|N_{\bar{G}}(v)|=n-k-1$ we have identified $2^{n-k-1}$ sets that are not in $\Upsilon(G,\bar{G})$.
%
%

No sets satisfy the construction of both $S$ and $T$ since $v\in S$ and $v\notin T$ and so we have $2^k+2^{n-k-1}$ sets that are not counted in $\Upsilon(G,\bar{G}))$.
We conclude $\Upsilon(G,\bar{G})\leq 2^n-(2^k+2^{n-k-1})$ and thus
\begin{equation*}
    \partial (G) + \partial (\bar{G})= 2^n + \Upsilon(G,\bar{G}) \leq 2^{n+1}-2^{k}-2^{n-k-1}.
\end{equation*}\end{proof}

To prove Theorem \ref{thm:main} we apply Lemma \ref{lem:k} for a vertex of degree at least $\floor{\frac{n}{2}}$, which must exist in either $G$ or $\bar{G}$. This eliminates the need for the knowledge of the degree of a specific vertex in $G$.

\begin{proof}[Proof of Theorem~\ref{thm:main}]
Let $G$ be a graph on $n$ vertices. Since $\max\{\Delta(G),\Delta(\bar{G})\}\geq \floor{\frac{n}{2}}$, there exists some vertex $v\in V(G)$ such that $\deg_G(v)=\floor{\frac{n}{2}}+d$ or $\deg_{\bar{G}}(v)=\floor{\frac{n}{2}}+d$ where $d\geq 0$. Without loss of generality suppose $\deg_G(v)=\floor{\frac{n}{2}}+d$ where $d\geq 0$. From Lemma \ref{lem:k} we have
  \[  \partial (G) + \partial (\bar{G})\leq
    2^{n+1}-2^{\floor{\frac{n}{2}}+d}-2^{n-\left(\floor{\frac{n}{2}}+d\right)-1}
    = 2^{n+1}-2^d\cdot 2^{\floor{\frac{n}{2}}}-\frac{2^{\ceil{\frac{n}{2}}-1}}{2^d}.\]
Considering the cases $d=0$ and $d>0$ separately we have
   \[\dom(G)+\dom(\bar{G})\leq 2^{n+1}-2^d\cdot 2^{\floor{\frac{n}{2}}}-\frac{2^{\ceil{\frac{n}{2}}-1}}{2^d}\leq 2^{n+1}-2^{\floor{\frac{n}{2}}}-2^{\ceil{\frac{n}{2}}-1}.\]
\end{proof}

\section{Degree Condition} \label{sec:DegreeCondition}
In this section we use Lemma \ref{lem:k} and our conjectured extremal graph to get a degree condition on all possible extremal graphs.
\begin{theorem}\label{thm:degbound}
If $G$ is a graph on $n$ vertices that maximizes $\partial(G)+\partial(\bar{G})$, then $\min\{\delta(G),\delta(\bar{G})\}\geq \floor{\frac{n}{2}}-2$ and $\max\{\Delta(G),\Delta(\bar{G})\}\leq \ceil{\frac{n}{2}}+1$.
\end{theorem}
\begin{proof}
Let $G$ be a graph on $n$ vertices such that $G$ maximizes $\partial(G)+\partial(\bar{G})$. First  suppose $n$ is even. Suppose that for some $v\in V(G)$, we have $\deg_G(v)\geq \frac{n}{2}+d$ for some integer $d\geq 2$. By Lemma \ref{lem:k} 
\begin{align*}
    \partial(G)+\partial(\bar{G})&\leq 2^{n+1}-2^{\frac{n}{2}+d}-2^{n-(\frac{n}{2}+d)-1}\\
    &=2^{n+1}-2^{d-1}\cdot2^{\frac{n}{2}+1}-\frac{2^{\frac{n}{2}+1}}{2^{d+2}}\\
    &<2^{n+1}-2\cdot 2^{\frac{n}{2}+1}\\
    &<2^{n+1}-2^{\frac{n}{2}+2}+4\\
    &=\partial\left(K_{\frac{n}{2},\frac{n}{2}}\right)+\partial\left(\overline{K_{\frac{n}{2},\frac{n}{2}}}\right).
\end{align*}
This contradicts that $G$ is extremal. Therefore, $\deg_G(v)\leq \frac{n}{2}+1$. The same argument applies for $\bar{G}$, so $\deg_{\bar{G}}(v)\leq \frac{n}{2}+1$. For any vertex $v$, $\deg_G(v)+\deg_{\bar{G}}(v)=n-1$ so these upper bounds imply $\deg_G(v)\geq n-(\frac{n}{2}+1)-1=\frac{n}{2}-2$ and $\deg_{\bar{G}}(v)\geq n-(\frac{n}{2}+1)-1=\frac{n}{2}-2$. These four inequalities imply the result when $n$ is even.

Now suppose $n$ is odd and that for some $v\in V(G)$, $\deg_G(v)\geq \frac{n+1}{2}+d$ where $d\geq 2$. By Lemma \ref{lem:k},
\begin{align*}
    \partial(G)+\partial(\bar{G})&\leq 2^{n+1}-2^{\frac{n+1}{2}+d}-2^{n-(\frac{n+1}{2}+d)-1}\\
    &=2^{n+1}-2^{d-1}\cdot 2^{\frac{n+3}{2}}-\frac{2^{\frac{n+1}{2}}}{2^{d+2}}\\
    &<2^{n+1}-2\cdot 2^{\frac{n+3}{2}}\\
    &<2^{n+1}- 2^{\frac{n+3}{2}}-2^{\frac{n+1}{2}}+4\\
    &=\partial\left(K_{\frac{n+1}{2},\frac{n-1}{2}}\right)+\partial\left(\overline{K_{\frac{n+1}{2},\frac{n-1}{2}}}\right).
\end{align*}
Again, this contradicts that $G$ is extremal. Therefore, $\deg_G(v)\leq \frac{n+1}{2}+1$. As before this implies $\deg_{\bar{G}}(v)\leq \frac{n+1}{2}+1$, $\deg_G(v)\geq n-(\frac{n+1}{2}+1)-1=\frac{n-1}{2}-2$ and $\deg_{\bar{G}}(v)\geq n-(\frac{n+1}{2}+1)-1=\frac{n-1}{2}-2$ which imply the result when $n$ is odd. 
\end{proof}
This theorem could be used in a future proof of Conjecture \ref{conj}, as it eliminates numerous graphs from consideration for each $n$.

\section{Conclusion}\label{sec:Conclusion}

There are several obstacles to proving Conjecture \ref{conj} using some traditional techniques.  One strategy would be to start with a graph and move edges between the graph and the complement in a way that increases $\dom(G)+\dom(\bar{G})$ at each edge move.  However there are several examples that show this isn't possible. For example, $\dom(C_5) + \dom(\overline{C_5}) = 42$, but moving any edge results in only $40$ dominating sets. Using a counting argument one can prove that
\begin{prop}\label{prop:multipartite}
For any complete multipartite graph $G$ on $n$ vertices that is not the complete balanced bipartite graph or its complement 
    \[\dom(G)+\dom(\bar{G}) < \dom(K_{\floor{\frac{n}{2}},\ceil{\frac{n}{2}}}) + \dom(\overline{K_{\floor{\frac{n}{2}},\ceil{\frac{n}{2}}}}).\]
\end{prop}
\noindent If one could show that any extremal graph should be a complete multipartite graph then Proposition \ref{prop:multipartite} would complete a proof of Conjecture \ref{conj}.

A proof of Conjecture \ref{conj} also doesn't work out nicely by induction on the number of vertices.  Let $H_n$ be the complete balanced bipartite graph on $n$ vertices, $G$ denote any graph on $n$ vertices and $G+v$ mean the addition of one vertex, $v$, and any edges we want. We might try to prove that 
\[(\dom(H_{n+1}) + \dom(H_{n+1})) - (\dom(H_n) + \dom(\overline{H_n}))>(\dom(G+v)+\dom(\overline{G+v})) - (\dom(G)+\dom(G)).\]
That is, the step from a maximal graph to the maximal graph on one more vertex increases the Nordhaus-Gaddum sum by more than adding a vertex to any other graph would. However, as one example, $G=K_{1,3}$ does not have this property.

Theorem \ref{thm:main} does give us a good result asymptotically.  To see this, consider how close $\dom(G)+\dom(\bar{G})$ can be to $2^{n+1}$ (a trivial upper bound). 
The complete balanced bipartite graph shows that 
    \[\max\{\dom(G) + \dom(\bar{G})\} \geq 2^{n+1} - 2^{\floor{n/2} + 1} - 2^{\ceil{n/2}+1} + 4\]
where the maximum is taken over all graphs $G$ on $n$ vertices.  This shows that the gap between $\max\{\dom(G)+\dom(\bar{G})\}$ and $2^{n+1}$ is at most
\begingroup
\addtolength{\jot}{.5em}
    \begin{align*}
    \biggl(4-o(1)\biggr)\,2^{n/2} &\text{\hspace{.2in} if $n$ is even}\\
    \biggl(3\sqrt{2} - o(1)\biggr)\,2^{n/2} &\text{\hspace{.2in} if $n$ is odd}
    \end{align*}
    \endgroup
and we conjecture this gap is the smallest possible.  From Theorem \ref{thm:main} we know that
    \[\max(\dom(G)+ \dom(\bar{G})) \leq 2^{n+1}-2^{\floor{n/2}} - 2^{\ceil{n/2}-1}.\]
which means that the gap is always at least
\begingroup
\addtolength{\jot}{.5em}
     \begin{align*}
    \biggl(\frac{3}{2}\biggr)\,2^{n/2} &\text{\hspace{.2in} if $n$ is even}\\
    \biggl(\sqrt{2}\biggr)\,2^{n/2} &\text{\hspace{.2in} if $n$ is odd}
    \end{align*}
    \endgroup
\noindent Thus, $2^{n/2}$ is the right order of magnitude for the gap between $2^{n+1}$ and $\max\{\dom(G)+\dom(\bar{G})\}$.

\section*{Acknowledgements}

The second author was supported by the Alayont Undergraduate Research Fellowship in Mathematics at Grand Valley State University. We would like to thank David Galvin for his contributions to the analysis in the Conclusion and Stefan Wagner for his helpful comments on a draft of this paper.

\bibliographystyle{unsrt}
\bibliography{DomSet}
\end{document}